\documentclass{amsart}

\newtheorem{theorem}{Theorem}[section]
\newtheorem{lemma}[theorem]{Lemma}

\newtheorem{corollary}[theorem]{Corollary}
\newtheorem{proposition}[theorem]{Proposition}
\usepackage{graphicx,amssymb}

\theoremstyle{definition}

\theoremstyle{remark}
\newtheorem{remark}[theorem]{Remark}

\numberwithin{equation}{section}

\begin{document}

\title[Left-orderability]{Left-orderability and exceptional Dehn surgery on two-bridge knots}

\author{Adam Clay}
\address{CIRGET, Universit\'{e} du Qu\'{e}bec, \'{a} Monter\'{e}al, Case Postale 8888,
Succursale Centre-ville, Montr\'{e}al QC, H3C 3P8.}
\email{aclay@cirget.ca}
\author{Masakazu Teragaito}
\address{Department of Mathematics and Mathematics Education, Hiroshima University,
1-1-1 Kagamiyama, Higashi-hiroshima, Japan 739-8524.}
\email{teragai@hiroshima-u.ac.jp}
\thanks{The first author is partially supported by an NSERC postdoctoral fellowship.
The second author is partially supported by Japan Society for the Promotion of Science,
Grant-in-Aid for Scientific Research (C), 22540088.
}%





\begin{abstract}
We show that any exceptional non-trivial Dehn surgery on a hyperbolic two-bridge knot,
yields a $3$-manifold whose fundamental group is left-orderable.
This gives a new supporting evidence for a conjecture of Boyer, Gordon and Watson.
\end{abstract}

\maketitle

\section{Introduction}

A group $G$ is \textit{left-orderable\/} if it admits a strict total ordering $<$, which 
is invariant under left-multiplication.
The fundamental groups of many $3$-manifolds, for example, all knot and link groups, are known to be left-orderable.
On the other hand, there are many $3$-manifolds whose fundamental groups are not left-orderable.
Since a left-orderable group is torsion-free, lens spaces provide such typical examples. 
There is a more general notion, called an $L$-space, introduced by Ozsv\'{a}th and Szab\'{o} \cite{OS}
in terms of Heegaard Floer homology.
These include lens spaces, elliptic manifolds, etc.
Recently, Boyer, Gordon and Watson \cite{BGW} conjectured that
a prime, rational homology $3$-sphere is an $L$-space if and only if
its fundamental group is not left-orderable.
This conjecture is verified for a few classes of $3$-manifolds \cite{BGW,G,I}.

In \cite{T}, 
the second author proved that any exceptional non-trivial Dehn surgery on a hyperbolic twist knot
yields a $3$-manifold whose fundamental group is left-orderable.
Since such a twist knot does not admit Dehn surgery yielding an $L$-space,
it gives a supporting evidence for the conjecture of Boyer, Gordon and Watson.

In the present paper, we examine the other hyperbolic two-bridge knots.
According to the classification of exceptional Dehn surgery on hyperbolic two-bridge knots \cite{BW},
it is sufficient to consider the following three cases; twist knots, 
$K[c_1,c_2]$\ ($c_1$ and $c_2$ are even, and $|c_1|, |c_2|>2$),
and $K[c_1,c_2]$\ ($c_1$ is odd, $c_2$ is even, and $|c_1|, |c_2|>2$).
Here, a two-bridge knot $K[c_1,c_2]$ corresponds to
a (subtractive) continued fraction 
\[
[c_1,c_2]^-=\cfrac{1}{c_1
          -\cfrac{1}{c_2}}
\]
in the usual way (\cite{HT}).  See also Section \ref{sec:pre}.
In particular, the double branched cover of the $3$-sphere $S^3$
branched over $K[c_1,c_2]$ is a lens space $L(c_1c_2-1,c_2)$.
The first case was settled in \cite{T}. 
For the second case, the only exceptional non-trivial surgery is $0$-surgery.
The resulting manifold is prime (\cite{Ga2}) and has positive Betti number, so its fundamental group is left-orderable \cite{BRW}.
For the last case, the only exceptional non-trivial surgery has slope $2c_2$, which
yields a toroidal manifold.
We settle this remaining case.

\begin{theorem}\label{thm:main}
Let $K$ be the two-bridge knot corresponding to a \textup{(}subtractive\textup{)} continued fraction $[c_1,c_2]$,
where $c_1$ is odd and $c_2$ is even, and $|c_1|, |c_2|>2$.
Then $2c_2$-surgery on $K$ yields a graph manifold whose fundamental group is left-orderable.
\end{theorem}

Hence this immediately implies the following.

\begin{corollary}\label{cor}
Let $K$ be a hyperbolic two-bridge knot.
Then any exceptional non-trivial Dehn surgery on $K$ yields
a $3$-manifold whose fundamental group is left-orderable.
\end{corollary}

We would expect that any non-trivial Dehn surgery on a hyperbolic 
two-bridge knot yields a 3-manifold whose fundamental group is 
left-orderable, but this is still a challenging problem.

\section{$L$-space surgery}\label{sec:pre}

Let $K$ be the two-bridge knot corresponding to $[c_1,c_2]$, satisfying the assumption of Theorem
\ref{thm:main}.  Set $c_1=2b_1+1$ and $c_2=2b_2$.
We can assume that $c_1>0$, so $b_1\ge 1$, and $|b_2|\ge 2$.
In Figure \ref{fig:knot}, a rectangular box means half-twists with indicated numbers.
They are right-handed if the number is positive, left-handed, otherwise.

\begin{figure}[ht]
\includegraphics[scale=0.6]{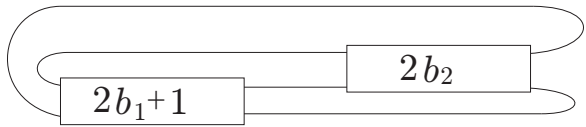}
\caption{The two-bridge knot $K[2b_1+1,2b_2]$}\label{fig:knot}
\end{figure}

Since $2$-bridge knots are alternating (\cite{Go}), we can invoke Theorem 1.5 of \cite{OS} to conclude that $2c_2$-surgery on $K$ does not yield and $L$-space.  However, we can argue this fact directly as follows.

\begin{lemma}\label{lem:fibered}
The knot $K=K[2b_1+1,2b_2]$ is fibered if and only if $b_1=1$ and $b_2>0$.
\end{lemma}

\begin{proof}
We have
\begin{equation*}
[2b_1+1,2b_2]^-=
\begin{cases}
[2b_1,\underbrace{-2,-2,\dots,-2}_{2b_2-1}]^- &\text{if $b_2>0$},\\
[2b_1+2,\underbrace{2,2,\dots,2}_{-2b_2-1}]^- &\text{if $b_2<0$}.
\end{cases}
\end{equation*}
This implies that a minimal genus Seifert surface of $K$ is
obtained by plumbing a single $2b_1$-twisted, or $(2b_1+2)$-twisted, annulus
with Hopf bands.
Then the conclusion immediately follows from \cite{Ga}.
\end{proof}

Recall that a rational homology $3$-sphere $Y$ is an \textit{$L$-space} if
its Heegaard Floer homology $\widehat{HF}(Y)$ has
rank equal to $|H_1(Y;\mathbb{Z})|$.

\begin{proposition}\label{prop:noL}
The knot $K$ does not admit an $L$-space surgery.
\end{proposition}

\begin{proof}
By \cite{N}, if $K$ is not fibered, then $K$ does not admit an $L$-space surgery.
Hence it is sufficient to examine the case where $b_1=1$ and $b_2>0$ by Lemma \ref{lem:fibered}.
Then, as seen in the proof of Lemma \ref{lem:fibered},
$K$ has genus $b_2$.

On the other hand,
the double branched cover of $S^3$ branched over $K$ is a lens space $L(6b_2-1,2b_2)$.
Hence the determinant $|\Delta_K(-1)|$ of $K$ equals to $6b_2-1$, where $\Delta_K(t)$ is the Alexander polynomial of $K$.

Suppose that $K$ admits an $L$-space surgery.
Then $\Delta_K(t)$ has a form of
\[
\Delta_K(t)=(-1)^k+\sum_{j=1}^k (-1)^{k-j}(t^{n_j}+t^{-n_j})
\]
for some sequence of positive integers $0<n_1<n_2<\cdots <n_k$ by \cite{OS}.
Since $K$ is fibered, its genus equals to $n_k$.
Thus $|\Delta_K(-1)|\le 2k+1\le 2n_k+1$.
Hence, $6b_2-1\le 2b_2+1$, a contradiction.
\end{proof}

\section{Fundamental group}

By using the Montesinos trick (\cite{M}), we will examine the structure of the resulting manifold
by $4b_2$-surgery on $K=K[2b_1+1,2b_2]$ to obtain a presentation of its fundamental group.

First, put the knot $K$ in a symmetric position as illustrated in Figure \ref{fig:symmetric}.
\begin{figure}[ht]
\includegraphics[scale=0.7]{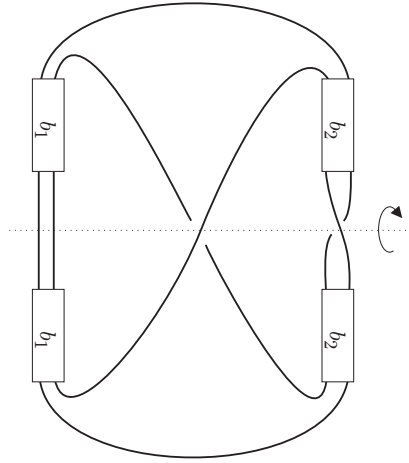}
\caption{$K$ in a symmetric position}\label{fig:symmetric}
\end{figure}
By taking a quotient under the involution, whose axis is indicated by a dotted line there,
we obtain a $2$-string tangle $\mathcal{T}$, which is drawn as the outside
of a small circle,  in Figure \ref{fig:monte1}.

\begin{figure}[htbp]
\includegraphics[scale=0.7]{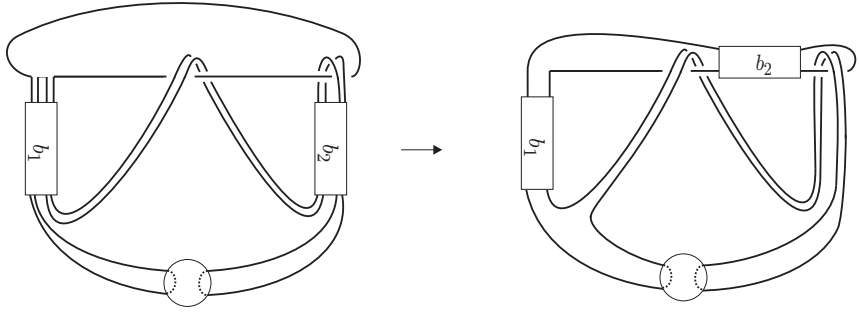}
\caption{Montesinos trick}\label{fig:monte1}
\end{figure}

If the $\infty$-tangle, which is indicated there, is filled into the small circle, then
we obtain a trivial knot.
This means that the double branched cover of the tangle $\mathcal{T}$
recovers the exterior of $K$.
We chose the framing so that the $0$-tangle filling corresponds to
$4b_2$-surgery on $K$ upstairs.
Figure \ref{fig:monte2} shows the resulting link by filling the $0$-tangle.
The link admits an essential Conway sphere $S$ depicted there.

\begin{figure}[htbp]
\includegraphics[scale=0.7]{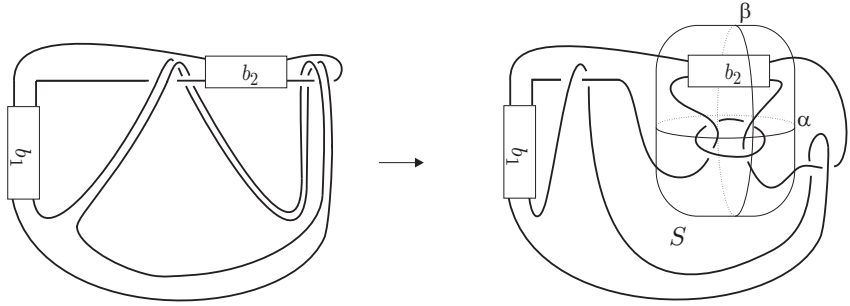}
\caption{The link after $0$-filling}\label{fig:monte2}
\end{figure}

Let $\mathcal{T}_1=(B_1,t_1)$ and $\mathcal{T}_2=(B_2,t_2)$ be the tangles defined by $S$, that
are located outside and inside of $S$, respectively.
Here, $t_1$ consists of two arcs, but $t_2$ consists of two arcs and a single loop.
Also, let $M_i$ be the double branched cover of the $3$-ball $B_i$ branched over $t_i$. 

\begin{lemma}
\label{lem:decompose}
\begin{enumerate} \noindent \item $M_1$ is the exterior of the torus knot of type $(2,2b_1+1)$.
The loops $\alpha$ and $\beta$ on $S$ lift to a meridian $\mu$ and a regular fiber $h$
of the exterior \textup{(}with the unique Seifert fibration\textup{)}, respectively.
\item $M_2$ is the union of the twisted $I$-bundle $KI$ over the Klein bottle and the
cable space $C$ of type $(b_2,1)$.
The loop $\alpha$ lifts to a regular fiber of the cable space $C$ with the unique Seifert fibration, and
a regular fiber of $KI$ with a Seifert fibration over the M\"{o}bius band.
\end{enumerate}
\end{lemma}

\begin{proof}
(1) By filling $\mathcal{T}_1$ with a rational tangle as in Figure \ref{fig:torusknot},
we obtain a trivial knot.
Then the core $\gamma$ of the filled rational tangle lifts to the torus knot of type $(2,2b_1+1)$.
This shows that $M_1$ is the exterior of the torus knot of type $(2,2b_1+1)$, and
$\alpha$ lifts to a meridian.

On the other hand, $\mathcal{T}_1$ is a Montesinos tangle whose double branched cover is
a Seifert fibered manifold over the disk with two exceptional fibers.
Moreover, $\beta$ lifts to a regular fiber (see \cite{BZ}).

\begin{figure}[htbp]
\includegraphics[scale=0.7]{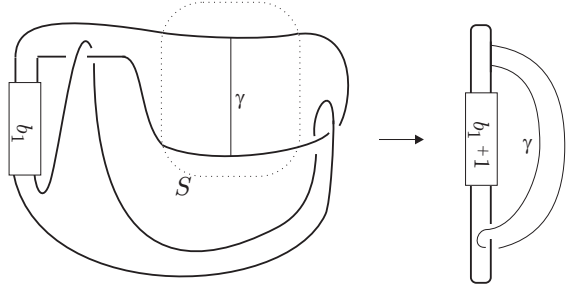}
\caption{$\mathcal{T}_1$ filled with a rational tangle}\label{fig:torusknot}
\end{figure}

(2) For $\mathcal{T}_2$, 
there is another essential Conway sphere $P$ as illustrated in Figure \ref{fig:ki}.

\begin{figure}[htbp]
\includegraphics[scale=0.8]{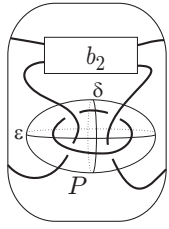}
\caption{Conway sphere $P$ in $\mathcal{T}_2$}\label{fig:ki}
\end{figure}

The inside of $P$ is a Montesinos tangle, whose double branched cover is
the twisted $I$-bundle $KI$ over the Klein bottle.
It is well known that $KI$ admits two Seifert fibrations; one over the disk
with two exceptional fibers, the other over the M\"{o}bius band with no exceptional fiber.
In fact, the loop $\delta$ (resp.~$\varepsilon$) on $P$ lifts
to a regular fiber of the former (resp.~latter) fibration.

The outside of $P$ lifts to the cable space of type $(b_2,1)$,
where $\alpha$ lifts to a regular fiber with respect to its unique fibration.
\end{proof}

\begin{remark}\label{rem:str}
In fact, $M_2$ admits a Seifert fibration over the M\"{o}bius band with
one exceptional fiber of index $|b_2|$.
Also, $M_2$ can be obtained by attaching a solid torus $J$ to the twisted $I$-bundle $KI$ over the
Klein bottle along annuli on their boundaries so that a regular fiber on $\partial (KI)$, with a Seifert fibration
over the M\"{o}bius band, runs $|b_2|$ times along a core of $J$.
\end{remark}

\begin{lemma}\label{lem:each}
For $M_1$, the fundamental group has a presentation
$\pi_1(M_1)=\langle a, b : a^2=b^{2b_1+1}\rangle$, with
a meridian $\mu=b^{-b_1}a$ and a regular fiber $h=a^2=b^{2b_1+1}$.
Also,
$\pi_1(M_2)=\langle x, y, z : x^{-1}yx=y^{-1}, y=z^{b_2}\rangle$.
\end{lemma}

\begin{proof}
For $M_1$, it is a standard fact, see \cite{BZ}.
For $M_2$, 
we first have $\pi_1(KI)=\langle x, y : x^{-1}yx=y^{-1}\rangle$,
where $x^2$ (resp.~$y$) represents a regular fiber of $KI$
with the Seifert fibration over the disk (resp.~M\"{o}bius band).
As in Remark \ref{rem:str},
decompose $M_2$ into $KI$ and a solid torus $J$ along an annulus.
Then $\pi_1(M_2)=\langle x, y, z : x^{-1}yx=y^{-1}, y=z^{b_2}\rangle$, where
$z$ represents a core of $J$ (with a suitable orientation).
\end{proof}

\begin{proposition}\label{pro:pi1}
Let $M$ be the resulting manifold by $4b_2$-surgery on $K$.
Then the fundamental group $\pi_1(M)$ has a presentation
\[
\pi_1(M)=\langle x, y, z, a, b : x^{-1}yx=y^{-1}, y=z^{b_2}, a^2=b^{2b_1+1}, \mu=y, h=zx^2\rangle,
\]
where $\mu=b^{-b_1}a$ and $h=a^2=b^{2b_1+1}$.
\end{proposition}

\begin{proof}
Let $\phi:\partial M_1\to \partial M_2$ be the identification map.
By Lemma \ref{lem:decompose}, $\phi(\mu)=y$.
Thus it is sufficient to verify that $\phi(h)=zx^2$.

Let $D_0$ be a disk with two holes, and let $c_0$ be the outer boundary component, and
$c_1$, $c_2$ the inner boundary components.
Then set $W=D_0\times S^1$.
We identify $D_0$ with $D_0\times \{*\}\subset D_0\times S^1$.
See Figure \ref{fig:compo}, where $W$ is obtained as the double branched cover
of the left tangle. 

\begin{figure}[htbp]
\includegraphics[scale=0.7]{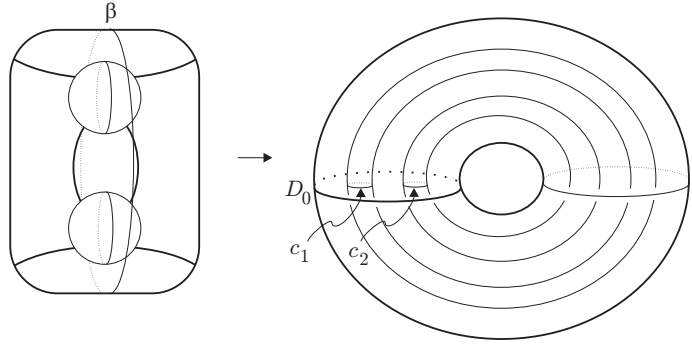}
\caption{$W=D_0\times S^1$}\label{fig:compo}
\end{figure}

Let $T_i=c_i\times S^1$.
Then $M_2$ is obtained from $W$ by attaching a solid torus $S^1\times D^2$ to $T_1$, $KI$ to $T_2$.
More precisely, $c_2$ 
is identified with a regular fiber of $KI$ with the Seifert fibration over
the disk.
Similarly, $c_1$ is identified with $S^1\times \{q\}\subset S^1\times \partial D^2$.

Since $c_0$ is a lift of the loop $\beta$ on $S$, 
$\phi(h)=c_0=c_1c_2$ with suitable orientations.
As above,
$c_1$ and $c_2$ correspond to $z$, $x^2$, respectively.
\end{proof}

\section{Left-orderings}

In this section, we prepare a few facts on left-orderings needed later.

Let $G$ be a left-orderable non-trivial group.
This means that $G$ admits a strict total ordering $<$ such that
$a<b$ implies $ga<gb$ for any $g\in G$. 
This is equivalent to the existence
of a \textit{positive cone\/} $P \ (\ne \varnothing)$, which is a semigroup and
gives a disjoint decomposition
$P\sqcup \{1\} \sqcup P^{-1}$.
For a given left-ordering $<$, the set
$P=\{g\in G\mid g>1\}$ gives a positive cone.
Any element of $P$ (resp.~$P^{-1}$) is said to be \textit{positive\/} (resp.~\textit{negative\/}).
Conversely, given a positive cone $P$,
declare $a<b$ if and only if $a^{-1}b\in P$.
This defines a left-ordering.

We denote by $\mathrm{LO}(G)$ the set of all positive cones in $G$.
This is regarded as the set of all left-orderings of $G$ as mentioned above.
For $g\in G$ and $P\in \mathrm{LO}(G)$, let $g(P)=gPg^{-1}$.
This gives a $G$-action on $\mathrm{LO}(G)$.
In other words,
for a left-ordering $<$ of $G$,
an element $g$ sends $<$ to a new left-ordering $<^g$ defined as follows:
$a<^g b$ if and only if $ag<bg$.
We say that $<$ and $<^g$ are \textit{conjugate\/} orderings.
Also, a family $L\subset \mathrm{LO}(G)$ is said to be \textit{normal\/} if it is $G$-invariant.

For $i=1,2$,
let $G_i$ be a left-orderable group and $H_i$ a subgroup of $G_i$, and
let $L_i\subset \mathrm{LO}(G_i)$ be a family of left-orderings.
Let $\phi:H_1\to H_2$ be an isomorphism.
We call that $\phi$ is \textit{compatible\/} for the pair $(L_1,L_2)$ if
for any $P_1\in L_1$, there exists $P_2\in L_2$ such that
$h_1\in P_1$ implies $\phi(h_1)\in P_2$ for any $h_1\in H_1$.

\begin{theorem}[Bludov-Glass \cite{BG}]\label{thm:BG}
For $i=1,2$,
let $G_i$ be a left-orderable group and $H_i$ a subgroup of $G_i$.
Let $\phi:H_1\to H_2$ be an isomorphism.
Then the free product with amalgamation $G_1*G_2\ (H_1\overset{\phi}{\cong} H_2)$
is left-orderable if and only if there exist normal families $L_i\subset \mathrm{LO}(G_i)$
for $i=1,2$ such that $\phi$ is compatible for $(L_1,L_2)$.
\end{theorem}

The next is well known.

\begin{lemma}\label{lem:extension}
Consider a short exact sequence of groups
\begin{equation}\label{eq:seq}
1\to K \to G \overset{\pi}{\to} H\to 1.
\end{equation}
Suppose $K$ and $H$ are left-orderable, with left-orderings $<_H$ and $<_K$, respectively.
For $g\in G$, declare that $1<g$ if $\pi(g)\ne 1$ and $1<_H \pi(g)$, or
if $\pi(g)=1$ and $1<_K g$.
Then this defines a left-ordering of $G$.
\end{lemma}

Suppose that we have a short exact sequence as (\ref{eq:seq}), where
$H$ is torsion-free and abelian.
Let $A$ be a subgroup of $G$ that is isomorphic to $\mathbb{Z}^2$.
We assume that $A\cap K=\langle x\rangle$ is an infinite cyclic group.
Since $H$ is torsion-free, the element $x$ is primitive in $A$, so we can choose
another element $y$ so that $\{x,y\}$ forms a basis of $A$.

Define two left-orderings $<_A$ and $<'_A$ of $A$ as follows:
\begin{itemize}
\item[(1)] Given $x^ry^s\in A$, $1<_{A}x^ry^s$ if $s>0$, else $s=0$ and $r>0$.
\item[(2)] Given $x^ry^s\in A$, $1<'_{A}x^ry^s$ if $s>0$, else $s=0$ and $r<0$.
\end{itemize}

\begin{lemma}\label{lem:restrict}
With notation as above, there exists a normal family $L\subset \mathrm{LO}(G)$
of left-orderings such that every left-ordering of $L$ restricts
to either $<_A$ or $<'_A$ on the subgroup $A$.
\end{lemma}

\begin{proof}
Choose a left-ordering $<_H$ of $H$ such that $1<_H \pi(y)$, and let
$<_K$ be an arbitrary left-ordering of $K$.
Construct a left-ordering $<$ of $G$ as in Lemma \ref{lem:extension},
using $<_H$ and $<_K$.
Then let $L\subset \mathrm{LO}(G)$ be the set of all conjugates of this ordering.
By construction, $L$ is normal.

Let $g\in G$ be an arbitrary element.
For $x^ry^s\in A$ with $s\ne 0$,
we have
\[
1<^g x^ry^s \Longleftrightarrow 1<g^{-1}x^ry^sg \Longleftrightarrow 1<_H \pi(g^{-1}x^ry^sg)=\pi(y)^s,
\]
since $H$ is abelian and $x\in K$.
From the choice of $<_H$, this happens only when $s>0$.
Thus $<^g$ restricts to $<_A$ or $<'_A$ on $A$, according as
$1<_K g^{-1}xg$ or $g^{-1}xg<_K 1$.
\end{proof}

\begin{remark}\label{rem:both}
The normal family $L$ obtained in Lemma \ref{lem:restrict}
contains both a left-ordering which restricts to $<_A$ on $A$ and
one which restricts to $<'_A$. 
For, if we have one, then the other is obtained by switching
the positive cone and negative cone.
\end{remark}

\section{Proof of Theorem \ref{thm:main}}

Let $G_1=\pi_1(M_1)=\langle a, b : a^2=b^{2b_1+1}\rangle$, with
a meridian $\mu=b^{-b_1}a$ and a regular fiber $h=a^2=b^{2b_1+1}$.
Then
\begin{eqnarray*}
G_{1} &=&\langle a,b,c : a^2=b^{2b_1+1}, c=ba^{-1}\rangle \\
    &=&\langle b,c : b=cb^{2b_1}c\rangle.
\end{eqnarray*}

Thus this is $\Gamma_{2b_1}$ in Navas's notation \cite{Na}.
Hence, we can assign Navas's left-ordering to $G_1$.

In \cite{T}, we show that

\begin{lemma}\label{lem:navas}
Let $<^g$ be a conjugate ordering of Navas's left-ordering $<$ of $G_1$. 
Assume $1<^g\mu^rh^s$. Then, 
\begin{itemize}
\item[(i)] $s>0$\textup{;} or
\item[(ii)] $s=0$ and 
$r>0$ \textup{(}resp.~$r<0$\textup{)} if $g^{-1}\mu g>1$ 
\textup{(}resp.~$g^{-1}\mu g<1$\textup{)}.
\end{itemize}
\end{lemma}


Next, we will examine $G_2=\pi_1(M_2)=\langle x, y, z : x^{-1}yx=y^{-1}, y=z^{b_2}\rangle$.
Since $G_2$ is the fundamental group of an irreducible $3$-manifold with toroidal boundary,
it is left-orderable \cite{BRW}.
Let $\pi:G_2\to \mathbb{Z}$ be a homomorphism defined by
$\pi(x)=1$, $\pi(y)=\pi(z)=0$.
Thus we have a short exact sequence
\[
1\to K \to G_2 \overset{\pi}{\to} \mathbb{Z} \to 1.
\]

Let $A$ be a rank two free abelian group generated by $\{y,zx^2\}$.
In fact, $A=\pi_1(\partial M_2)$.
Then $A\cap K=\langle y\rangle$. 
Hence by Lemma \ref{lem:restrict}, we have a normal family $L\subset \mathrm{LO}(G_2)$ 
such that any left-ordering in $L$ restricts to $<_A$ or $<'A$ on $A$, which are defined as follows:
\begin{itemize}
\item[(1)] Given $y^r(zx^2)^s\in A$, $1<_{A}y^r(zx^2)^s$ if $s>0$, else $s=0$ and $r>0$.
\item[(2)] Given $y^r(zx^2)^s\in A$, $1<'_{A}y^r(zx^2)^s$ if $s>0$, else $s=0$ and $r<0$.
\end{itemize}


\begin{proof}[Proof of Theorem \ref{thm:main}]
Let $L_1\subset \mathrm{LO}(G_1)$ be the set of all conjugate orderings of Navas's left-ordering of $G_1$.
This is normal by definition.
Let $L_2\subset \mathrm{LO}(G_2)$ be the normal family given above.

Recall that the identification map $\phi:\partial M_1\to \partial M_2$
is given by
$\phi(\mu)=y$ and $\phi(h)=zx^2$.
To show that $\pi_1(M)$ is left-orderable,
it is sufficient to verify that $\phi$ is compatible for the pair $(L_1,L_2)$
by Theorem \ref{thm:BG}.

For a left-ordering $<^g\in L_1$,
suppose $1<^g \mu^rh^s$.
If $1<^g \mu$, then 
$s>0$, or $s=0$ and $r>0$ by Lemma \ref{lem:navas}.
Since $\phi(\mu^rh^s)=y^r(zx^2)^s$, we choose a left-ordering in $L_2$, which
restricts to $<_A$ on $A$.
Similarly, if $\mu<^g 1$, then choose a left-ordering in $L_2$, which
restricts to $<'_A$ on $A$.
This shows that $\phi$ is compatible for $(L_1,L_2)$.
\end{proof}

\bibliographystyle{amsplain}

\begin{thebibliography}{BGW}
\bibitem{BG}
V. V. Bludov and A. M. W. Glass,
\textit{Word problems, embeddings, and free products of right-ordered groups with amalgamated subgroup},
Proc. Lond. Math. Soc. (3) \textbf{99} (2009), 585--608.

\bibitem{BGW}
S. Boyer, C. McA. Gordon and L. Watson,
\textit{On $L$-spaces and left-orderable fundamental groups},
preprint, \texttt{arXiv:1107.5016}.

\bibitem{BRW}
S. Boyer, D. Rolfsen and B. Wiest,
\textit{Orderable 3-manifold groups},
Ann. Inst. Fourier (Grenoble) \textbf{55} (2005), 243--288.

\bibitem{BW}
M. Brittenham and Y. Q. Wu,
\textit{The classification of exceptional Dehn surgeries on 2-bridge knots},
Comm. Anal. Geom. \textbf{9} (2001), 97--113.

\bibitem{BZ}
G. Burde and H. Zieschang,
\textit{Knots}, de Gruyter Studies in Mathematics, 5. Walter de Gruyter \& Co., Berlin, 2003.


\bibitem{Ga}
D. Gabai,
\textit{The Murasugi sum is a natural geometric operation},
Low-dimensional topology (San Francisco, Calif., 1981), 131--143, 
Contemp. Math., 20, Amer. Math. Soc., Providence, RI, 1983. 

\bibitem{Ga2}
D. Gabai,
\textit{Foliations and the topology of 3-manifolds. III},
J. Differential Geom. \textbf{26} (1987), 479--536.

\bibitem{Go}
R.E. Goodrick,
\textit{Two bridge knots are alternating knots},
Pacific J. Math. (3) \textbf{40} (1972), 561--564.

\bibitem{G}
J. Greene,
\textit{Alternating links and left-orderability},
preprint, \texttt{arXiv:1107.5232}. 

\bibitem{HT}
A. Hatcher and W. Thurston, 
\textit{Incompressible surfaces in 2-bridge knot complements},
Invent. Math. \textbf{79} (1985), 225--246. 

\bibitem{I}
T. Ito,
\textit{Non-left-orderable double branched coverings},
preprint, \texttt{arXiv:1106.1499}. 

\bibitem{M}
J. M. Montesinos,
\textit{Surgery on links and double branched covers of $S^{3}$},
Knots, groups, and 3-manifolds (Papers dedicated to the memory of R. H. Fox), pp. 227--259.
Ann. of Math. Studies, No. 84, Princeton Univ. Press, Princeton, N.J., 1975. 

\bibitem{Na}
A. Navas,
\textit{A remarkable family of left-ordered groups: central extensions of Hecke groups},
J. Algebra \textbf{328} (2011), 31--42,

\bibitem{N}
Y. Ni,
\textit{Knot Floer homology detects fibred knots},
Invent. Math. \textbf{170} (2007), 577--608.

\bibitem{OS}
P. Ozsv\'{a}th and Z. Szab\'{o},
\textit{On knot Floer homology and lens space surgeries},
Topology \textbf{44} (2005), 1281--1300. 

\bibitem{T}
M. Teragaito,
\textit{Left-orderability and exceptional Dehn surgery on twist knots},
preprint, \texttt{arXiv:1109.2965}.

\end{thebibliography}

\end{document}